\documentclass{amsart}
\usepackage{filecontents}
\usepackage{amsfonts,amssymb,amsmath,amsthm}
\usepackage{url}
\usepackage{enumerate}
\usepackage[utf8]{inputenc}
\usepackage[pagebackref]{hyperref}
\usepackage{ amssymb }
\usepackage{enumitem}
\usepackage{amsthm}
\usepackage{cleveref}
\usepackage{mathrsfs}  
\usepackage{chngcntr}
\usepackage[sort&compress,square,comma,numbers]{natbib}
\usepackage{xcolor}

\newtheorem{theorem}{Theorem}
\newtheorem{lemma}{Lemma}[theorem]
\newtheorem{corollary}{Corollary}[theorem]
\newtheorem{prop}{Proposition}
\crefname{lemma}{lemma}{lemmas}
\crefname{prop}{proposition}{propositions}

\theoremstyle{remark}
\newtheorem*{remark}{Remark}

\counterwithin*{equation}{section}

\newcommand{\lbstoc}[0]{\underline{B}}

\newcommand{\abs}[1]{\ensuremath{\left\lvert#1\right\rvert}}
\renewcommand{\P}[1]{\ensuremath{\mathbb{P}(#1)}}
\newcommand{\bracket}[1]{\ensuremath{\left(#1\right)}}
\newcommand{\expect}[2][]{\ensuremath{\mathbb{E}_{#1}\left[#2\right]}}
\newcommand{\expcond}[2]{\ensuremath{\mathbb{E}\left[#1\middle\rvert #2\right]}}
\newcommand{\ball}[2]{\ensuremath{B(#1,#2)}}
\newcommand{\proj}[1]{\ensuremath{\text{proj}_{#1}}}
\newcommand{\pderiv}[2]{\ensuremath{\frac{\partial{#1}}{\partial{#2}}}}
\newcommand{\Cmik}[0]{\ensuremath{C}} 
\newcommand{\lip}[1]{\text{Lip}(#1)} 
\newcommand{\dbound}[0]{C^\infty_\text{db}} 
\newcommand{\convex}[0]{c} 
\newcommand{\flip}[1]{\hat{#1}} 
\DeclareMathOperator*{\argmin}{arg\,min}

\title{On Optimal Stochastic Ballistic Transports}
\author{Alistair Barton \and Nassif Ghoussoub}
\thanks{This work is part of a Master's thesis prepared by A. Barton under the supervision of N. Ghoussoub. Both authors were partially supported by  the 
Natural Sciences and Engineering Research Council of Canada (NSERC)}
\address{
Mathematics Department\\
University of British Columbia\\
Vancouver, BC}
\email{arsbar@math.ubc.ca, nassif@math.ubc.ca}

\begin{document}
\begin{abstract}
For a given Lagrangian $L:[0,T]\times M\times M^\ast\rightarrow \mathbb{R}_+$ and probability measures $\mu\in\mathcal{P}(M^\ast)$, $\nu\in \mathcal{P}(M)$, we introduce the stochastic ballistic transportation problems 
\begin{align}\tag{$\star$}
    \underline{B}(\mu,\nu):=\inf\left\{\mathbb{E}\left[\langle V,X_0\rangle +\int_0^T L(t,X,\beta(t,X))\,dt\right]\middle\rvert V\sim\mu,X_T\sim \nu\right\}\\\tag{$\star\star$}
    \overline{B}(\nu,\mu):=\sup\left\{\mathbb{E}\left[\langle V,X_T\rangle -\int_0^T L(t,X,\beta(t,X))\,dt\right]\middle\rvert V\sim\mu,X_0\sim \nu\right\}
\end{align}
where $X$ is a diffusion process with drift $\beta$. This cost is based on the stochastic optimal transport problem presented by Mikami and the deterministic ballistic transport introduced by Ghoussoub. We obtain a Kantorovich-style duality result that reformulates this problem in terms of solutions to the Hamilton-Jacobi-Bellman equation
\begin{equation*}
    \frac{\partial\phi}{\partial t}+\frac{1}{2}\Delta \phi+H(t,x,\nabla\phi)=0,
\end{equation*}
and show how optimal processes may be thereby attained.
\end{abstract}
\maketitle

\section{Introduction \& Related works}

 In its modern incarnation, the problem of optimal transportation conceived by Monge \cite{Monge}  is formulated as follows: let $X,Y$ be measure spaces with a cost function $c:X\times Y\rightarrow \mathbb{R}$. For two probability measures $\mu\in\mathcal{P}(X)$, $\nu\in\mathcal{P}(Y)$ the minimum cost of transportation between them is then
\begin{equation}\label{eq:OTclassic}
    C(\mu,\nu)=\inf\left\{\int_{X\times Y} c(x,y)\,d\pi(x,y)\middle\rvert \pi\in\mathcal{T}(\mu,\nu)\right\},
\end{equation}
where $\mathcal{T}(\mu,\nu)$ is the set of admissable \emph{transport plans} for $(\mu,\nu)$. A transport plan $\pi$ is admissable if it is a probability measure on $X\times Y$ with first marginal $(\proj{1})_\#\pi=\mu$ and second marginal $(\proj{2})_\#\pi=\nu$.
\par 
A well known cost on $\mathbb{R}^n\times \mathbb{R}^n$ is the Wasserstein cost $c(x,y):=\langle x,y\rangle$ studied by Brenier \cite{Brenier}, whose associated transportation costs are
\begin{align}\tag{$\underline{W}$}
    \underline{W}(\mu,\nu):=&\inf\left\{\int_{\mathbb{R}^n\times \mathbb{R}^n} \langle x,y\rangle \,d\pi(x,y)\middle\rvert \pi\in\mathcal{T}(\mu,\nu)\right\},\\
    \tag{$\overline{W}$}
    \overline{W}(\mu,\nu):=&\sup\left\{\int_{\mathbb{R}^n\times \mathbb{R}^n} \langle x,y\rangle \,d\pi(x,y)\middle\rvert \pi\in\mathcal{T}(\mu,\nu)\right\}.
\end{align}
These problems are clearly related by noting that $\underline{W}(\mu,\nu)=-\overline{W}(\flip{\mu},\nu)$ where $\flip{\mu}$ is the reflection of $\mu$ (i.e.,  $\flip{\mu}(A):=\mu(-A)$). Probabilistically, we can rewrite these as:
\begin{align*}
    \underline{W}(\mu,\nu):=&\inf_{\pi}\left\{\expect{ \langle X,Y\rangle }\middle\rvert X\sim \mu,Y\sim\nu\right\},\\
    \overline{W}(\mu,\nu):=&\sup_\pi\left\{\expect{ \langle X,Y\rangle }\middle\rvert \rvert X\sim \mu,Y\sim\nu\right\},
\end{align*}
respectively. Note that, in general, $(X,Y)$ are not independent. Indeed Brenier showed that in the case where $\mu,\nu$ are absolutely continuous with respect to the Lebesgue measure then, in the optimal case for $\underline{W}$, $Y$ is completely determined by $X$ (in particular, there exists a convex function $\phi$ such that $Y=\nabla\phi(X)$) \cite{Brenier}.
\par 
The \emph{dual} formulation, discovered by Kantorovich  \cite{kantorovich}, allows us to equate the \emph{primal} problem considered in \cref{eq:OTclassic} to a problem on the space of functions:
\begin{equation}\label{eq:gendual}
    C(\mu,\nu)=\sup\left\{\int_X f(x)\,d\mu(x)+\int_Y g(y)\,d\nu(y)\middle\rvert (f,g)\in\mathcal{K}(c)\right\}.
\end{equation}
Here the set of \emph{Kantorovich potentials} $\mathcal{K}(c)$ is the set of functions $(f,g)\in L^1(d\mu)\times L^1(d\nu)$ such that $f(x)+g(y)\le c(x,y)$. It is clear that 
the Kantorovich potentials may be assumed to satisfy
\begin{align*}
    f(x)=&\inf_y\{c(x,y)-g(y)\}& 
    g(y)=&\inf_x\{c(x,y)-f(x)\}.
\end{align*}
For Wasserstein costs, this amounts to $f$ and $g$ being Legendre duals of each other: $g(y)=\tilde{f}(y)$ and $f(x)=\tilde{g}(x)$ for $\underline{W}$, and $g(y)=f^{\ast}(y)$ and $f(x)=g^\ast(x)$ for $\overline{W}$. Here $\tilde{f}$ (tesp., $f^*$) denote the concave (resp., convex) Legendre transform of $f$.
\par 
In this paper, we will investigate the ballistic stochastic dynamic transportation problem. This is a variant of the transportation problem developed by combining the ballistic and stochastic dynamic problems that we will briefly summarize for contextual purposes.
\par
Bernard and Buffoni \cite{BernardBuffoni} introduced the idea of a dynamic cost function derived from the idea of minimizing the total action associated with a Lagrangian $L:[0,T]\times TM\rightarrow\mathbb{R}_+$ in the transportation between two measures $\nu_0,\nu_T$ on a manifold $M$ (here $TM$ is the tangent bundle on the manifold $M$). This was motivated to connect Mather theory with optimal transport. Here the cost $c_s^t(x,y)$ associated with the transportation problem denoted $C_s^t(\nu_s,\nu_t)$ is given by
\begin{equation}\label{eq:dyncost}
    c_s^t(x,y):=\inf\left\{\int_s^t L(t,\gamma(t),\dot{\gamma}(t))\,dt\middle\rvert \gamma\in C^1([s,t],M), \gamma(s)=x, \gamma(t)=y\right\},
\end{equation}
where $s<t$ and $(s,t)\in [0,T]^2$. 
They show \cite[Proposition 17]{BernardBuffoni} that this problem has a dual formulation akin to \cref{eq:gendual}, where the Kantorovich potentials may be restricted to the set $\{(u(x,T),-u(x,0))\rvert u\in HJ\}$ where $HJ$ is the set of continuous viscosity solutions of the Hamilton-Jacobi equation
\begin{equation}\label{eq:HJ}\tag{HJ}
    \pderiv{u}{t}+H(t,x,\nabla u)= 0,
\end{equation}
and $H(t,x,p):=\sup_v\{\langle p,v\rangle -L(t,x,v)\}$ is the hamiltonian associated with $L$. They further showed \cite[Theorem B]{BernardBuffoni} that, in the case where the measures are absolutely continuous, the optimal transportation measure can be described as the Hamiltonian flow of an initial momentum measure.
\par 
The ballistic cost was recently introduced by Ghoussoub \cite{Ghoussoub}. This is a cost on $M^\ast\times M$ (where $M=\mathbb{R}^n$) taking the form:
\begin{equation}
    b_s^t(v,x):=\inf\left\{\langle v, \gamma (s)\rangle +\int_s^t L(t,\gamma(t),\dot{\gamma}(t))\,dt\middle\rvert \gamma\in C^1([s,t],M), \gamma(t)=x\right\},
\end{equation}
This was used to lift the Hopf-Lax formula to Wasserstein space. Ghoussoub also  showed that this cost corresponds to the dual problem
\begin{equation*}
    \underline{B}_c(\mu_0,\nu_T)=\sup\left\{\int_M V(T,x)\,d\nu_T+\int_{M^*}\tilde{V}_0(v)\,d\mu_0\middle\rvert V\in HJ_c\right\}
\end{equation*}
where $HJ_c$ is the set of solutions to the Hamilton-Jacobi equation (\ref{eq:HJ}) with initial condition $u(x,0)$ concave. Here $\tilde{V}_0(v):=\inf_{x\in M}\{\langle v,x\rangle +V_0(x)\}$ is the concave legendre transform. This result was obtained through the following interpolation result
\begin{equation*}
    \underline{B}_c(\mu_0,\nu_T)=\inf\left\{\underline{W}(\mu_0,\nu_0)+C_c(\nu_0,\nu_T)\middle\rvert \nu_0\in\mathcal{P}(M)\right\}.
\end{equation*}
This is akin to the Hopf-Lax formula on Wasserstein space for the functional $\mathcal{U}^{\mu_0}:\rho\mapsto \underline{W}(\mu_0,\rho)$ under the Lagrangian
\begin{equation*}
    \mathcal{L}(t,\rho,v)=\int_M L(t,x,v(x))\,d\rho(x)
\end{equation*}
where $v$ is a vector field related to $\rho$ via the continuity equation $\frac{d\rho}{dt}+{\rm div} (\rho v)=0$. (See Villani
\cite[Chapter 8]{villani}).
\par 
The following stochastic variant of the action based transportation cost was considered by Mikami and Thieullen \cite{mikami}. 
\begin{equation}\label{eq:stoc}
    C_S^\epsilon(\nu_0,\nu_T):=\inf\left\{\expect{\int_0^T L(t,X,\beta_X(t,X))\,dt}\middle\rvert X\in \mathcal{A}_{\nu_0}^{\nu_T}(\epsilon),X_T\sim\nu_T\right\}.
\end{equation}
The set of \emph{transportation processes} $\mathcal{A}_{\nu_0}^{\nu_T}(\epsilon)$ may be understood as the set of stochastic processes solving the stochastic differential equation $dX=\beta_X(t,X)\,dt+\epsilon dW_t$ for some $\beta_X(t,X)$, where $W_t$ is $\sigma(X_s:0\le s\le t)$-Brownian motion, initially distributed according to the measure $\nu_0$ and distributed at time $T$ according to $\nu_T$. An earlier result from Mikami \cite{Mikamilsc} shows that the deterministic case considered by Brenier \cite{Brenier}, i.e., when $L(x, v)=\frac{1}{2}|v|^2$, 
is recovered from the corresponding stochastic transport as one takes $\epsilon\rightarrow 0$.

This problem cannot be formulated in the same way as deterministic transportation problems (i.e., \cref{eq:OTclassic})---indeed, the corresponding cost function could be $c(x,y)=C(\delta_x,\delta_y)=\infty$ everywhere even under reasonable assumptions on the Lagrangian.
Therefore, classical Monge-Kantorovich duality need not apply here. 
Nevertheless, Mikami and Thieullen \cite[Theorem 2.1]{mikami} derived the following dual formulation through Legendre transforms considerations: 
\begin{equation*}
    C(\nu_0,\nu_T):=\sup\left\{\int_M \phi(T,x)\,d\nu_T(x)-\int_M \phi(0,x)\,d\nu_0(x)\middle\rvert \phi\in HJB\right\}.
\end{equation*}
Here, $HJB$ is the set of classical solutions to the following Hamilton-Jacobi-Bellman equation
\begin{equation}\tag{HJB}
    \pderiv{\phi}{t}+\frac{\epsilon}{2}\Delta \phi+H(t,x,\nabla \phi)=0,
\end{equation}
with final condition smooth and bounded. Furthermore for any optimal process $X$ there is a sequence of solutions $\phi_n$ to (\ref{eq:HJB}) such that the SDE \begin{equation*}
    dX=\lim_{n\rightarrow\infty}\nabla_p H(t,X,\nabla\phi_n(t,x))dt+\epsilon dW_t
\end{equation*}is satisfied---this is the counterpart to the Hamiltonian flow for the stochastic cost.
\par 
In this paper we aim to combine the results of Ghoussoub, and Mikami and Thieullen, by considering the stochastic ballistic transportation problem:
\begin{equation*}
    \underline{B}(\mu_0,\nu_T):=\inf\left\{\expect{\langle V,X(0)\rangle +\int_0^T L(t,X,\beta_X(t,X))\,dt}\middle\rvert V\sim\mu_0, X(\cdot)\in \mathcal{A}^{\nu_T}\right\}.
\end{equation*}
This may be considered in either the framework of stochastic control theory, or a stochastic version of the Hopf-Lax equation in measure space.

\section{Assumptions \& Main Results}
\subsection{Assumptions}
We will operate on the following assumptions on the Lagrangian, that reduce to those considered by Mikami and Thieullen in \cite{mikami}:
\begin{enumerate}[label=(A\arabic*),start=0]
    \item $(t,x,v)\mapsto L(t,x,v)$ is $\mathcal{C}^3$, always positive, and has $\nabla_v^2L(t,x,v)>0$. 
    \item There exists $\delta>1$ such that 
    \begin{equation*}
        \liminf_{\abs{u}\rightarrow\infty}\frac{\inf_{t,x} L(t,x,u)}{\abs{u}^\delta}>0.
    \end{equation*}
    (In \cref{prop:A1convex} we show that this is equivalent $L(t,x,v)$ being bound below by a convex function $\ell(v)\in\Omega(\abs{v}^\delta)$).\\
    \item[]
    \begin{equation}\tag{A2}\hspace{1mm}
        \Delta L(\epsilon_1,\epsilon_2):=\sup\left\{\frac{1+L(t,x,u)}{1+L(s,y,u)}-1\middle\rvert\abs{t-s}<\epsilon_1,\abs{x-y}<\epsilon_2\right\}\overset{\epsilon_1,\epsilon_2\rightarrow0}{\longrightarrow} 0.
    \end{equation}
    \stepcounter{enumi}
    \item
    \begin{enumerate}[label=(\roman*),start=1]
     \item $\sup_{t,x} L(t,x,0)<\infty$.
     \item $\abs{\nabla_x L(t,x,v)}/(1+L(t,x,v))$ is bounded.
     \item $\sup\left\{\abs{\nabla_v L(t,x,u)}: \abs{u}\le R\right\}<\infty$ for all $R$.
    \end{enumerate}
    \item (i) $\Delta L(0,\infty)<\infty$ or (ii) $\delta =2$ in (A1).
\end{enumerate}
We give a brief summary of the role of the above assumptions:
\\
(A1) is a coercivity result that is necessary for sequential compactness of minimizing transportation processes (\cref{prop:Ccoerc}).\\ 
(A2) is used to show that the expected action of a transportation process is lower semi-continuous \cite[eqs. (3.17),(3.38)-(3.41)]{Mikamilsc}.\\ 
(A3) is used with (A0) to derive the Hamilton-Jacobi-Bellman equation in \cref{prop:HJB} \cite[p. 210, Remark 11.2]{HJB}.\\ 
(A4,i) allows us to uniformly bound the ratio $\frac{1+L(t,x,u)}{1+L(t,y,u)}$, while (A4,ii) ensures the minimizing $\beta_X$ satisfies $\int_0^T \abs{\beta_X(t,X)}^2\,dt<\infty$ a.s., permitting us to assume the process $X(\cdot)$ is absolutely continuous with respect to $W(\cdot)$ \cite[Theorem 7.16]{LipsterStatofRP}. Either of these results with (A0)-(A3) is sufficient to show convexity of $(\nu_0,\nu_T)\mapsto C(\nu_0,\nu_T)$ \cite[Lemma 3.2]{mikami}.
\begin{remark}\label{prop:A1convex}
(A1) is equivalent to $L(t,x,v)$ being bound below by a convex function $\ell(v)$ such that 
\begin{equation*}
    \liminf_{\abs{v}\rightarrow\infty} \frac{\ell(v)}{\abs{v}^\delta}>0.
\end{equation*}
Indeed, if $L$ is bound by such an $\ell$, it is simple to see (A1) holds. We prove the reverse direction by construction. By (A1), we may assume that there exists a $U\in\mathbb{R}$ such that for all $\abs{u}>U$
\begin{equation*}
        \frac{\inf_{t,x} L(t,x,u)}{\abs{u}^\delta}>\alpha,\qquad 
        (\alpha>0,\delta>1),
\end{equation*}
defining $\alpha,\delta$. Then the convex function $\ell:M^\ast\rightarrow \mathbb{R}$ defined by
\begin{equation*}
    \ell(v):=\begin{cases}0& \abs{v}<U\\ 
    \alpha\abs{\abs{v}-U}^\delta& \abs{v}\ge U
    \end{cases}
\end{equation*}
is a lower bound on $L(t,x,v)$ and is asymptotically bounded below by $\abs{v}^\delta$.
\end{remark}
\subsection{Notation}
The space of probability measures on a space $X$  will be indicated by $\mathcal{P}(X)$, while the subset of measures with finite barycenter will be denoted $\mathcal{P}_1(X):=\{\mu\in\mathcal{P}(X): \int \abs{x}\,d\mu<\infty\}$. We will work on the space $M:=\mathbb{R}^d$ however we will still refer to the dual space $M^\ast$ explicitly, for clarity. Measures in $\mathcal{P}(M)$ will be designated by $\nu$, while $\mu$ will designate measures in $\mathcal{P}(M^\ast)$.
\par Given a Lagrangian $L:[0,T]\times M\times M^\ast\rightarrow\mathbb{R}$ satisfying properties (A), we define the (stochastic) dynamic and ballistic variants of the optimal transportation problem mentioned in the introduction:
\begin{align*}
    C(\nu_0,\nu_T):=&\inf\left\{ \expect{\int_0^T L(t,X(t),\beta_X(t,X(t)))\,dt}\middle\rvert X(\cdot)\in \mathcal{A}_{\nu_0}^{\nu_T}\right\},\\
    \lbstoc(\mu_0,\nu_T):=&\inf\left\{\expect{\langle Y,X(0)\rangle+ \int_0^T L(t,X(t),\beta_X(t,X(t)))\,dt}\middle\rvert Y\sim \mu_0,X(\cdot)\in \mathcal{A}^{\nu_T}\right\}.
\end{align*}
Here $\mathcal{A}_{\nu_0}^{\nu_T}$ refers to the set of $\mathbb{R}^d$-valued continuous semimartingales $X(\cdot)$---initially distributed according to $\nu_0$ and finally distributed according to $\nu_T$---such that there exists a measurable drift $\beta_X:[0,T]\times C([0,T])\rightarrow M^\ast$ where
\begin{enumerate}[label=(\roman*)]
    \item $\omega\mapsto\beta_X(t,\omega)$ is $\mathcal{B}(C([0,t]))_+$-measurable for all $t$.
    \item $W(t):=X(t)-X(0)-\int_0^t \beta_X(s,X)\,ds$ is a $\sigma[X(s):s\in[0,t]]$-Brownian motion.
\end{enumerate}
$\mathcal{A}^{\nu_T}$ refers to the same processes, except with initial distribution unconstrained, likewise $\mathcal{A}_{\nu_0}$ with the final distribution unconstrained and $\mathcal{A}$ with initial and final distribution unconstrained.

For convenience, we define the expected action of stochastic processes $X(\cdot)\in\mathcal{A}$:
\begin{equation*}
    \mathscr{A}(X):=\expect{\int_0^T L(t,X(t),\beta_X(t,X(t)))\,dt}.
\end{equation*}
We will also consider another, analogous, ballistic stochastic cost defined by:
\begin{equation*}
    \overline{B}(\nu_0,\mu_T):=\sup\left\{\expect{\langle V,X(T)\rangle- \int_0^T L(t,X(t),\beta_X(t,X(t)))\,dt}\middle\rvert V\sim\mu_T,X(\cdot)\in \mathcal{A}_{\nu_0}\right\},
\end{equation*}
which we will refer to as the \emph{maximizing ballistic cost}---in comparison, $\underline{B}$ may be referred to as the \emph{minimizing ballistic cost}.\par
We recall the convex and concave Legendre transforms, which we will denote
\begin{align*}
    f^\ast(v):=&\sup_{x\in M}\{\langle v,x\rangle -f(x)\},\\ 
    \tilde{f}(v):=&\inf_{x\in M}\{\langle v,x\rangle-f(x)\}=-(-f)^\ast(-v),
\end{align*}
respectively, where $v$ is in the dual space $M^\ast$. This notation will also be used for the legendre duals of functions of measures, such as $C_{\nu_0}:\nu\mapsto C(\nu_0,\nu)$ where the dual space is a linear subspace of continuous functions.
\par 
A set of functions that is of particular importance in this regard is the set of smooth functions with bound derivatives---notably, this may be realized as the set of mollified Lipschitz functions---which we will denote $\dbound$, and the subset of convex such functions $\convex\dbound$.
\subsection{Main Results}
We present an interpolation result that allows us to interpret the ballistic cost in terms of the Wasserstein cost and the stochastic action cost.
\begin{theorem}[Interpolation of $\underline{B}$]
If $L$ satisfies the assumptions (A), then our stochastic ballistic cost may be written as:
\begin{equation*}
    \lbstoc(\mu_0,\nu_T)=\inf_\nu\{\underline{W}(\mu_0,\nu)+C(\nu,\nu_T)\}.
\end{equation*}
This minimum is attained in the case where $\mu_0\in\mathcal{P}_1(M^\ast)$ and $\nu_T\in\mathcal{P}_1(M)$.
\end{theorem}
This leads to our main result, the dual of the stochastic ballistic cost:

\begin{theorem}[Dual of $\underline{B}$]
If $\nu_T\in\mathcal{P}_1(M)$ and $\mu_0$ with compact support are such that $\lbstoc(\mu_0,\nu_T)<\infty$, and the Lagrangian satisfies (A0)-(A4) then we have
\begin{equation}
    \lbstoc(\mu_0,\nu_T)=\sup_{-f\in \convex\lip{M}}\left\{\int_M f(x)\, d\nu_T(x)+\int_{M^*} \widetilde{\phi^f}(0,v)\,d\mu_0(v)\right\},
\end{equation}
where $\widetilde{g}$ denotes the concave legendre transform of $g$ and $\phi^{f}$ is the solution to the Hamilton-Jacobi-Bellman equation
\begin{align}\label{eq:HJB}\tag{HJB}
    \pderiv{\phi}{t}+\frac{1}{2}\Delta\phi(t,x)+H(t,x,\nabla\phi)=&0, & \phi(T,x)=&f(x).
\end{align}
\end{theorem}
This formulation of the problem allows us to write optimal processes in terms of a maximizing sequence of $\phi$:

\begin{corollary}[Optimal Processes for $\underline{B}$]
Under the assumptions of Theorem 2, with $d\mu_0\ll d\lambda$, we have that $(V,X(t))$ is a minimizing process for $\underline{B}(\mu_0,\nu_T)$ if and only if it is a solution to the SDE
\begin{align*}
    dX =& \nabla_p\lim_{n\rightarrow\infty} H(t,X,\nabla\phi_n(t,X))\, dt + dW_t\\
    X(0) =& \nabla\lim_{n\rightarrow\infty}\bracket{\widetilde{\phi_n^0}}(V),
\end{align*}
where $\phi_n$ is a sequence of solutions to (\ref{eq:HJB}) that is maximizing for the dual problem. (Here we use the notation $\phi^t_n(x):=\phi_n(t,x)$).
\end{corollary}
\subsubsection{Maximizing Cost}
We obtain parallel results for the maximizing cost $\overline{B}$:
\begin{theorem}[Interpolation of $\overline{B}$]
If $L$ matches assumptions (A), then
\begin{equation*}
    \overline{B}(\nu_0,\mu_T)=\sup\{\overline{W}(\nu,\mu_T)-C(\nu_0,\nu)\},
\end{equation*}
where, again, the maximum is attained when $\nu_0\in\mathcal{P}_1(M)$ and $\mu_T\in\mathcal{P}_1(M^\ast)$.
\end{theorem}
\begin{theorem}[Dual of $\overline{B}$]
Assume the Lagrangian satisfies the assumptions (A). If $\nu_0\in\mathcal{P}_1(M)$, $\mu_T$ has compact support, and $\overline{B}(\nu_0,\mu_T)<\infty$, then 
\begin{equation*}
    \overline{B}(\nu_0,\mu_T)=\inf_{g\in\convex \dbound}\left\{\int_{M^*} g^\ast\,d\mu_T+\int_M\phi^{g}\,d\nu_0\right\},
\end{equation*}
where $\phi^g$ solves the Hamilton-Jacobi-Bellman equation
\begin{align}\tag{HJB2}
    \pderiv{\phi}{t}+\frac{1}{2}\Delta \phi -H(t,x,\nabla\phi)=&0&\phi(x,T)=g(x)
\end{align}
\end{theorem}
Once more we may identify optimal processes by this formulation:
\begin{corollary}[Optimal Processes for $\overline{B}$]
If the assumptions of \cref{thm:Boverdual} are satisfied with $d\mu_T\ll d\lambda$, then $(V,X(t))$ is a maximizing process for $\overline{B}(\nu_0,\mu_T)$ iff it is a solution to the SDE
\begin{align*}
    dX=&\lim_{n\rightarrow\infty}\nabla_p H(t,x,\nabla\phi_n(t,X))\,dt+dW_t\\ 
    V=&\lim_{n\rightarrow\infty}\nabla\phi_n^T(X(T)), 
\end{align*}
for a sequence of solutions $\phi_n$ to (\ref{eq:HJB2}) that is minimizing for the dual problem.
\end{corollary}

\section{Proofs}
\setcounter{theorem}{0}
We first show a coercivity result that allows us to explore the structure of the problem and some of the consequences of our assumptions. This proposition is useful for showing existence of minimizers to our interpolation formula in \cref{thm:interpol}.

\begin{prop}[Coercivity of $C$]\label{prop:Ccoerc}
For any fixed $\nu_T\in\mathcal{P}_1(M)$, $N\in\mathbb{R}$, the set of measures $\nu\in \mathcal{P}_1(M)$ satisfying 
\begin{equation}\label{eq:tight}
    C(\nu,\nu_T)\le N\int \abs{x}\,d\nu(x)
\end{equation}is tight. In other words,  for every $\epsilon>0$ there exists an $R$ such that if $\nu\in \mathcal{P}_1(M)$ satsifies \cref{eq:tight}), then $\nu(\ball{R}{0}^c)\leq \epsilon$. 
\end{prop}
\begin{proof}
First we define the set of measures $\mathcal{T}_{\epsilon,R}:=\{\nu\in\mathcal{P}_1(M):\nu(\ball{R}{0}^c)>\epsilon\}$. We will assume $\nu\in\mathcal{T}_{\epsilon,R}$ for what follows.
\par 
We begin by defining $R_1$ to be such that $\nu_T(\ball{R_1}{0}^c)<\epsilon$. Let $X\in\mathcal{A}$ be a stochastic process with $X(0)\sim \nu$ and $X(T)\sim \nu_T$. We leave $R$ to be defined later, but note that if we define the set $\Omega_R:=\{ \abs{X(0)}>R\}$, then our assumption on $\nu$ yields $\mathbb{P}(\Omega_R)>\epsilon$. By positivity of $L$, this allows us to say that $\mathscr{A}(X)\ge \mathscr{A}(1_{\Omega_R}X)$ (henceforth we define the process $Y(t):=1_{\Omega_R}(\omega)X(t)$). Furthermore we can assume the probability of a path originating at a distance beyond $R$ from the origin and terminating within distance $R_1$ from the origin is
\begin{equation}\label{eq:crossbound}\begin{split}
    \P{\{\abs{X(0)}>R\}\cap\{\abs{X(T)}\le R_1\}}\ge &1-\P{\abs{X(0)}\le R}-\P{\abs{X(T)}> R_1}\\ 
    =& \P{\abs{X(0)}>R}-\P{\abs{X(T)}> R_1}\\ 
    =&\nu(\ball{R}{0}^c)-\nu_T(\ball{R_1}{0}^c)> \epsilon/2.
\end{split}
\end{equation}
\par 
Applying (A1) via our remark in \cref{prop:A1convex}, we assume that there is a convex function $\ell:M^\ast\rightarrow\mathbb{R}$ such that for all $\abs{u}>U$
\begin{equation}\label{eq:A1coerc}
        \frac{\ell(u)}{\abs{u}^\delta}>\alpha,\qquad 
        (\alpha>0,\delta>1),
\end{equation}
defining $\alpha,\delta$. Recall that $\ell(\abs{v})$ is a lower bound on $L(t,x,v)$.
\par 
Examining the process $Y$, we can take advantage of the convexity of $\ell$ to apply Jensen's inequality:
\begin{equation}\begin{split}\label{eq:Jensencoerc}
    \expect{\int_0^T L(t,Y,\beta_Y(t,Y))\,dt}\ge\,& \expect{\int_0^T \ell(\abs{\beta_Y(t,Y)})\,dt}\overset{(\text{J})}{\ge}\expect{\ell(\abs{V})T}\\
    \overset{(\ref{eq:A1coerc})}{>}&\alpha T\expect{1_{\abs{V}>U}\abs{V}^\delta}\ge\alpha T\bracket{\expect{\abs{V}^\delta}-U^\delta},
\end{split}
\end{equation}
where $\beta_Y:=1_{\Omega_R}\beta_X$ is the drift associated with the process $Y$ and $V:=(Y(T)-Y(0))/T$ is its time-average.
\par 
Putting this all together, we get
\begin{equation}\label{eq:Coerctogether}\begin{split}
    \mathscr{A}(X)\ge& \mathscr{A}(Y)\overset{(\ref{eq:Jensencoerc})}{>} \alpha T\bracket{\expect{\abs{V}^\delta}-U^\delta}=\alpha T\bracket{\expect{\abs{\frac{Y(T)-Y(0)}{T}}^\delta }-U^\delta}\\ 
    \overset{(\triangle)}{>}&\frac{\alpha}{T^{\delta-1}}\bracket{\expect{\abs{\abs{Y(0)}-\abs{Y(T)}}^\delta}-\bracket{UT}^\delta}
\end{split}
\end{equation}
Minimizing the expectation over $(Y(0),Y(T))$ is a type of well known optimal transport problem, with the optimal plan given by the monotone Hoeffding-Frechet mapping $x\mapsto G_{\abs{Y(0)}}(G_{\abs{Y(T)}}^{-1}(x))$, where $G_Z(t):=\inf\{x\in\mathbb{R}:t\ge\P{Z\le x}\}$ is the quantile function associated with the random variable $Z$ \cite[Theorem 1.1]{beiglbock}. Thus the optimal plan maps each quantile in one measure to the corresponding quantile in the other.
\par 
Focussing on the expectation in \cref{eq:Coerctogether}, we get
\begin{equation}\begin{split}\label{eq:finalcoerce}
    \expect{\abs{\abs{Y(0)}-\abs{Y(T)}}^\delta}\ge &\int\abs{x-y}^\delta\,d\bracket{\bracket{G_{\abs{Y(0)}}\times G_{\abs{Y(T)}} }_\#\lambda_{[0,1]}}(x,y)\\
    \overset{(\text{J})}{\ge}& \abs{\int x\,d\bracket{G_{\abs{Y(0)}}}_\#\lambda_{[0,1]}(x)-b}^\delta\\
    =& \bracket{\int_{\ball{R}{0}^c}\abs{x}\,d\nu(x)-b}^\delta,
\end{split}
\end{equation}
where $b:=\int\abs{x}\,d\nu_T/\epsilon>\expect{\abs{Y(T)}}$ and we require that $R>b/\epsilon$. We thus want to find $R$ such that
\begin{equation*}
    \frac{\alpha}{T^{\delta-1}} \bracket{\bracket{\int_{\ball{R}{0}^c}\abs{x}\,d\nu(x)-b}^\delta-\bracket{UT}^\delta}>N\bracket{\int_{\ball{R}{0}^c}\abs{x}\,d\nu(x)+R(1-\epsilon)}\ge N\int\abs{x}\,d\nu(x)
\end{equation*}
Letting $I_\nu(R):=\int_{\ball{R}{0}}\abs{x}\,d\nu(x)$ (which is at least $R\epsilon$ for all $\nu\in\mathcal{T}_{\epsilon,R}$), we find the condition
\begin{equation*}
    \frac{\alpha}{T^{\delta-1}}  \bracket{\bracket{I_\nu(R)^{1-\frac{1}{\delta}}-bI_\nu(R)^{-\frac{1}{\delta}}}^\delta-\frac{\bracket{UT}^\delta}{I_\nu(R)}}>N\bracket{1+\frac{R(1-\epsilon)}{I_\nu(R)}},
\end{equation*}
which by using the mentioned bound on $I_\nu(R)$ is satisfied if $R$ satisfies
\begin{equation*}
    \frac{\alpha}{T^{\delta-1}} \bracket{\bracket{(R\epsilon)^{1-\frac{1}{\delta}}-b(R\epsilon)^{-\frac{1}{\delta}}}^\delta-\frac{\bracket{UT}^\delta}{R\epsilon}}>\frac{N}{\epsilon}.
\end{equation*}
Taking $R\rightarrow\infty$ satisfies this inequality, implying this set is non-empty and intersects with our earlier constraints that $R>R_1+UT$ and $R>b/\epsilon$.
\end{proof}
\begin{remark}
The same argument holds if we add any fixed constant to the right side of \cref{eq:tight}.\par 
Furthermore, \cref{eq:finalcoerce} can be used to show that if $\nu_1\in\mathcal{P}_1(M)$ and $\nu_0\in\mathcal{P}(M)\setminus\mathcal{P}_1(M)$, then $C(\nu_0,\nu_1)=C(\nu_1,\nu_0)=\infty$. This leads to our focus on distributions with finite first moment in the following.
\end{remark}

We now state some properties of $C$ shown by Mikami:
\begin{prop}\label{prop:Cdualsc}(\cite[Theorem 2.1]{mikami}) 
If the Lagrangian satisfies (A1)-(A4), then 
\begin{equation*}
    C(\nu_0,\nu_T)=\sup\left\{\int_M f(x)\,d\nu_T-\int_M \phi^f(0,x)\,d\nu_0\middle\rvert f\in\mathcal{C}_b^\infty\right\},
\end{equation*}
where $\phi_f$ solves (\ref{eq:HJB}). Furthermore, $(\mu,\nu)\mapsto C(\mu,\nu)$ is convex and lsc under the weak topology.
\end{prop}
\subsection{Minimizing Ballistic Cost}
Now that we have demonstrated the coercivity and lower semicontinuity of $C$, we can move on to investigating the minimizing ballistic cost: $\underline{B}$. These previous results are sufficient to show the existence of a minimizer to the interpolation:

\begin{theorem}[Interpolation of $\underline{B}$]\label{thm:interpol}
If $L$ satisfies the assumptions (A), then our stochastic ballistic cost may be written
\begin{equation}\label{eq:interpol}
    \lbstoc(\mu_0,\nu_T)=\inf_\nu\{\underline{W}(\mu_0,\nu)+C(\nu,\nu_T)\}.
    \end{equation}
Furthermore, this minimum is attained in the case where $\mu_0\in\mathcal{P}_1(M^\ast)$ and $\nu_T\in\mathcal{P}_1(M)$.
\end{theorem}

\begin{proof}
First, to show that $\underline{B}(\mu_0,\nu_T)$ is greater or equal to its interpolations, consider $(Y_n,X_n(\cdot))$ such that
\begin{equation*}
    \expect{\langle Y_n,X_n(0)\rangle +\int_0^T L(t,X_n(t),\beta_{X_n}(X_n,t))\,dt}< \lbstoc(\mu_0,\nu_T)+\tfrac{1}{n}.
\end{equation*}
Let $\nu_n:=PX_n(0)^{-1}$, then 
\begin{align*}
    \underline{W}(\mu_0,\nu_n)\le&\expect{\langle Y_n,X_n(0)\rangle}, \\
    \Cmik(\nu_n,\nu_T)\le&\expect{\int_0^T L(t,X_n(t),\beta_{X_n}(X_n,t))\,dt},
\end{align*}
so 
\begin{equation*}
    \inf_\nu\left\{\underline{W}(\mu_0,\nu)+\Cmik(\nu,\nu_T)\right\}\le\underline{W}(\mu_0,\nu_n)+\Cmik(\nu_n,\nu_T)< \lbstoc(\mu_0,\nu_T)+\tfrac{1}{n}.
\end{equation*}
Taking $n\rightarrow\infty$ gives the inequality.\par 
To get the reverse inequality, take $\nu_n$ to be a sequence such that 
\begin{equation*}
    \underline{W}(\mu_0,\nu_n)+\Cmik(\nu_n,\nu_T)<\inf_\nu\left\{\underline{W}(\mu_0,\nu)+\Cmik(\nu,\nu_T)\right\}+\tfrac{1}{n}.
\end{equation*}
We may then define an admissible stochastic process $Z_n(\cdot)\in\mathcal{A}_{\nu_n}^{\nu_T}$ such that
\begin{equation*}
    \expect{\int_0^T L(t,Z_n(t),\beta_{Z_n}(Z_n,t))\,dt}< \Cmik(\nu_n,\nu_T)+\tfrac{1}{n}.
\end{equation*}
Then, let $W_n(\omega)\sim \gamma_{Z_n(0)(\omega)}^n$, where $d\gamma_x^n(y)\otimes d\mu_0(x)=d\gamma_n(x,y)$ is the disintegration of a measure $\gamma_n$ such that
\begin{equation*}
    \int \langle y,x\rangle \,d\gamma_n(x,y)<\underline{W}(\mu_0,\nu_n)+\tfrac{1}{n}.
\end{equation*}
Thus
\begin{equation*}\begin{split}
    \lbstoc(\mu_0,\nu_T)\le&\langle W_n,Z_n\rangle+\int_0^T L(t,Z_n(t),\beta_{Z_n}(Z_n,t))\,dt<\underline{W}(\mu_0,\nu_n)+\Cmik(\nu_n,\nu_T)+\tfrac{2}{n}\\
    <&\inf_\nu\left\{\underline{W}(\mu_0,\nu)+\Cmik(\nu,\nu_T)\right\}+\tfrac{3}{n}.
\end{split}
\end{equation*}
Lastly, to show this minimum is achieved in the case where $\mu_0$ and $\nu_T$ have finite mean, we must show that the minimizing sequence $\nu_n$ is sequentially compact in the weak topology. To do this, we note that the set of measures $\nu$ such that
\begin{equation*}
    C(\nu,\nu_T)<N\int\abs{y}\,d\nu(y)+\lbstoc(\mu_0,\nu_T)+1
\end{equation*}
is tight by \cref{prop:Ccoerc}. If we let $N:=\int \abs{x}\,d\mu_0(x)$, then the collection of measures such that
\begin{equation*}\begin{split}
    \lbstoc(\mu_0,\nu_T)+1>&C(\nu,\nu_T)+\underline{W}(\mu_0,\nu)\\
    >&C(\nu,\nu_T)-\int\abs{x}\abs{y}\,d\mu_0(x)\,d\nu(y)\\ 
    \overset{(\text{F})}{=}&C(\nu,\nu_T)-N\int\abs{y}\,d\nu(y),
\end{split}
\end{equation*}
is tight, where (F) is application of Fubini's theorem. Thus, by Prokhorov's theorem our minimizing sequence of interpolating measures necessarily weakly converges to a minimizing measure.
\end{proof}
\begin{remark}
The attainment of a minimizing $\nu_0$ of \cref{eq:interpol} is sufficient to show that of a minimizer of $\underline{B}$ when it is finite, since the minimum is achieved for both $\underline{W}$ and $C$ \cite[Proposition 2.1]{mikami}.
\end{remark}

\begin{corollary}[Properties of $\lbstoc$]\label{cor:Blsc}
If the Lagrangian satisfies (A), then: 
\begin{enumerate}[label=\alph*)]
    \item $\nu\mapsto\lbstoc(\mu_0,\nu)$ is lower semi-continuous on $\mathcal{P}_1(M)$ for all $\mu_0\in\mathcal{P}_1(M^\ast)$.
    \item $(\mu_0,\nu_T)\mapsto \lbstoc(\mu_0,\nu_T)$ is jointly convex.
\end{enumerate}
\end{corollary}
\begin{proof}[Proof of a)]
Consider a sequence $\nu_{T,n}\in\mathcal{P}_1(M)$ that converges weakly to $\nu_T\in\mathcal{P}_1(M)$. We have
\begin{equation*}
    \lbstoc(\mu_0,\nu_{T,n})=\underline{W}(\mu_0,\nu_n)+C(\nu_n,\nu_{T,n}).
\end{equation*}
for a sequence of measures $\nu_n\in\mathcal{P}_1(M)$ by \cref{thm:interpol}. Since $C$ is lsc (\cref{prop:Cdualsc}), we get
\begin{equation*}
    \liminf_{n\rightarrow\infty} \underline{W}(\mu_0,\nu_n)+C(\nu_n,\nu_{T,n})\ge \liminf_{n\rightarrow\infty} \underline{W}(\mu_0,\nu_n)+C(\nu_n,\nu_{T})\ge \lbstoc(\mu_0,\nu_T).
\end{equation*}
\end{proof}
\begin{lemma}\label{lem:convexlemma}
If $(x,y)\mapsto h(x,y)$ jointly convex, then $\bar{h}(x):=\inf_y h(x,y)$ is convex.
\end{lemma}
\begin{proof}
Consider $x_1,x_2\in X$. Then there exist sequences $y^n_1,y^n_2\in Y$ such that 
\begin{equation*}
    \lim_n h(x_i,y_i^n)=\bar{h}(x_i).
\end{equation*}
Now take $x_\lambda:=\lambda x_1+(1-\lambda)x_2$, and $y^n_\lambda:=\lambda y_1^n+(1-\lambda)y_2^n$, then
\begin{equation*}
    \bar{h}(x_\lambda)\le h(x_\lambda,y_\lambda^n)\le \lambda h(x_1,y_1^n)+ (1-\lambda) h(x_2,y_2^n)\overset{n\rightarrow\infty}{\longrightarrow} \lambda \bar{h}(x_1)+ (1-\lambda) \bar{h}(x_2)
\end{equation*}
\end{proof}
\begin{proof}[Proof of \cref{cor:Blsc}b)]
Define the function
\begin{equation}\label{eq:convexinterpol}
    h((\mu_0,\nu_T),\nu):=\underline{W}(\mu_0,\nu)+\Cmik(\nu,\nu_T).
\end{equation}
Then by \cref{thm:interpol} 
\begin{equation}\label{eq:convexlemma}
    \lbstoc(\mu_0,\nu_T)=\inf_\nu \left\{h((\mu_0,\nu_T),\nu)\right\}.
\end{equation}
But $h((\mu_0,\nu_T),\nu)$ is jointly convex---as $(\mu_0,\nu)\mapsto\underline{W}(\mu_0,\nu)$ and $(\nu,\nu_T)\mapsto \Cmik(\nu,\nu_T)$ are both jointly convex (\cref{prop:Cdualsc}). Hence applying \cref{lem:convexlemma} gives the joint convexity of \cref{eq:convexlemma}.
\end{proof}
Now that we have shown that $\nu_T\mapsto \underline{B}(\mu_0,\nu_T)$ is convex and lsc, we are able apply convex duality to get our dual formulation.

\begin{theorem}[Dual of $\underline{B}$]\label{thm:stocbdual}
If $\nu_T\in\mathcal{P}_1(M^\ast)$ and $\mu_0$ with compact support are such that $\lbstoc(\mu_0,\nu_T)<\infty$, and the Lagrangian satisfies (A0)-(A4), then we have
\begin{equation}\label{eq:mindual}
    \lbstoc(\mu_0,\nu_T)=\sup_{-f\in \convex\lip{M}}\left\{\int f(x)\, d\nu_T(x)+\int \widetilde{\phi^f}(0,v)\,d\mu_0(v)\right\},
\end{equation}
where $\widetilde{g}$ is the concave legendre transform of $g$ and $\phi^{f}$ is the solution to the Hamilton-Jacobi-Bellman equation
\begin{align}\label{eq:HJB}\tag{HJB}
    \pderiv{\phi}{t}+\frac{1}{2}\Delta\phi(t,x)+H(t,x,\nabla\phi)=&0, & \phi(T,x)=&f(x).
\end{align}
\end{theorem}
\begin{lemma}\label{lem:Wlegendre}
For $\mu\in\mathcal{P}(M)$ with compact support, define $\underline{W}_\mu:\mathcal{M}_1(M^\ast)\rightarrow \mathbb{R}$, (where $\mathcal{M}_1(M^\ast)$ is the linear space of measures $\nu$ on $M^\ast$ such that $\int_{M^*} \bracket{1+\abs{x}}\,d\mu<\infty$) to be
\begin{equation*}
    \underline{W}_{\mu}(\nu):=\begin{cases} \underline{W}(\mu,\nu)&\nu\in\mathcal{P}_1(M)\\
    \infty&\text{otherwise}.
    \end{cases}
\end{equation*}Then, for $f\in \lip{M}$, the convex Legendre transform of $\underline{W}_\mu$ is given by,
\begin{equation*}
    \underline{W}_{\mu}^\ast(f)=-\int \widetilde{f}\,d\mu.
\end{equation*}
\end{lemma}
The Lipschitz condition is a consequence of the dual space of $\mathcal{M}_1(M)$ being the space of Lipschitz functions $\lip{M}$ \cite{lipschalg}.
\begin{proof}
It is well known that $\underline{W}_\mu$ is convex and lsc (e.g. via the dual formulation, \cref{eq:gendual}). Thus, by convex analysis, $\underline{W}_\mu=\underline{W}_\mu^{\ast\ast}$, giving us:
\begin{equation*}
    \underline{W}_\mu(\nu)=\sup_{f\in\lip{M}}\left\{\int f\,d\nu -\underline{W}_\mu^\ast(f)\right\}.
\end{equation*}
But Kantorovich duality (when $\mu$ has compact support) also gives us
\begin{equation*}
    \underline{W}_\mu(\nu)=\sup_{f\in\lip{M}}\left\{\int f\,d\nu +\int \widetilde{f}\,d\mu\right\}=F_\mu^\ast(\nu),
\end{equation*}
where $F_\mu:f\mapsto -\int \widetilde{f} \,d\mu=\int (-f)^\ast \,d\hat{\mu}(v)$ (recall $d\hat{\mu}(v)=d\mu(-v)$). To show that $\underline{W}_\mu^\ast(f)=F_\mu(f)$ it suffices to show that $f\mapsto F_\mu(f)$ is convex and lower semicontinuous, in which case $\underline{W}_\mu^\ast(\nu)=F_\mu^{\ast\ast}(f)=F_\mu(f)$. We will do this by showing $g\mapsto g^\ast(v)$ is a convex mapping for all $v\in M^\ast$: Fix $\lambda\in[0,1]$ and $v\in M^\ast$ arbitrary.
\begin{equation*}
\begin{split}
    (\lambda g_1+(1-\lambda)g_2)^\ast (v)=&\sup_x \{\langle x,v\rangle -\lambda g_1(x)-(1-\lambda)g_2(x)\}\\
    \overset{(1)}{\le} &\sup_x \{\langle x,v\rangle -\lambda \bracket{\langle x,v\rangle -g^\ast_1(v)}-(1-\lambda)\bracket{\langle x,v\rangle -g^\ast_2(v)}\}\\
    =& \lambda g_1^\ast(v)+(1-\lambda)g_2^\ast(v),
\end{split}
\end{equation*}
where we use $g(x)\ge \langle x,v'\rangle-g^\ast (v')$ for all $v'\in M^\ast$ in (1)---taking the liberty to let $v'=v$.\par 
To show $F_\mu$ is weakly lsc, consider $f_n$ that converge pointwise to $f$ (a weaker condition than weak convergence). Then
\begin{equation*}\begin{split}
    \liminf_{n\rightarrow\infty}  F_\mu(f_n)=&\liminf_{n\rightarrow\infty} \int (-f_n)^\ast\,d\hat{\mu}\overset{\text{(F)}}{\ge}\int \liminf_{n\rightarrow\infty} \sup_x \{\langle v,x\rangle +f_n(x)\}\,d\hat{\mu}\\
    \ge & \int \sup_x \{\langle v,x\rangle +\liminf_{n\rightarrow\infty} f_n(x)\}\,d\hat{\mu}=\int (-f)^\ast\,d\hat{\mu}= F_\mu(f),
\end{split}
\end{equation*}
where we use $\{(-f_n)^\ast\}$ being bounded below by $\inf_n\{f_n(0)\}>-\infty$ to apply Fatou's lemma (F) and use $\sup g\ge g$ to get to the second line.\end{proof}

\begin{remark}
Note that we can limit Kantorovich duality to Lipschitz functions only because $\mu$ is compact, hence the optimal $f$ has bounded gradient. We could extend this result to non-compactly supported $\mu$ if we consider $\alpha$-H\"older continuous functions for any $\alpha>1$ rather than Lipschitz functions. This case is precluded from us here because of the constraint to functions with bounded derivatives in the following proposition.

Since the double dual of $\underline{W}_\mu$ is equivalent to the dual formulation, Brenier's theorem applies. And the maximizing $f$ provides the optimal transport $\nabla f$ of the primal problem.
\end{remark}
We now state a similar result regarding the dynamic cost:
\begin{prop}\label{prop:HJB}
Assume (A) and let $f\in \dbound$. Then, the unique solution to \cref{eq:HJB} is given by: 
\begin{equation}\tag{HJB'}\label{eq:dynprog}
    \phi^f(t,x)=\sup_{X\in\mathcal{A}}\left\{\expcond{f(X(T))-\int_t^T L(s,X(s),\beta_X(s,X))\,ds}{X(t)=x}\right\}.
\end{equation}
Moreover, there exists an optimal process $X$ with drift $\beta_X(t,X)=\argmin_v\{v\cdot \nabla \phi(t,x)+L(t,x,v)\}$.
\end{prop}
\begin{remark}
This proposition may be shown by removing the boundedness condition in \cite[Remark IV.11.2]{HJB}. In fact it suffices to only have the first three derivatives bounded. Note that integrating \cref{eq:dynprog} over $d\nu_0$ yields the legendre transform of $\nu_T\mapsto C(\nu_0,\nu_T)$ for $f\in\dbound$. A key step in proving \cref{thm:stocbdual} is showing that it is not necessary to consider functions outside of this set.\par 
That this proposition does not necessarily hold for functions with unbounded derivatives confines us to $\mu$ with bounded support as we saw in \cref{lem:Wlegendre}.
\end{remark}
\begin{proof}[Proof of \cref{thm:stocbdual}]
For $\mu_0\in \mathcal{P}_1(M^\ast)$, define the function $\underline{B}_{\mu_0}:\mathcal{M}_1(M)\rightarrow\mathbb{R}\cup\{\infty\}$ to be
\begin{equation*}
    \underline{B}_{\mu_0}(\nu):=\begin{cases}\underline{B}(\mu_0,\nu)&\nu\in\mathcal{P}_1(M)\\
    \infty &\text{otherwise}.\end{cases}
\end{equation*}
From \cref{cor:Blsc}, we get that $\underline{B}_{\mu_0}$ is convex and lsc. Thus, by convex analysis, we have
\begin{equation}\label{eq:Bddual}
    \underline{B}_{\mu_0}(\nu)=\underline{B}_{\mu_0}^{\ast\ast}(\nu)=\sup_{f\in \lip{M}}\left\{\int f\, d\nu-\underline{B}_{\mu_0}^\ast(f)\right\}.
\end{equation}
We break this into two steps. First we show that when $f\in \dbound$ the dual is appropriate:
\begin{align}\notag
    \underline{B}_{\mu_0}^\ast(f)&\,:=\sup_{\nu_T\in \mathcal{P}_1(M)}\left\{\int f\,d\nu_T-\underline{B}(\mu_0,\nu_T)\right\}\\\label{eq:dualexp}
    &\,\overset{(\ref{eq:interpol})}{=}\sup_{\substack{\nu_T\in\mathcal{P}_1(M)\\\nu\in\mathcal{P}_1(M)}}\left\{\int f\,d\nu_T-C(\nu,\nu_T)-\underline{W}(\mu_0,\nu)\right\}\\\notag
    &\overset{(\ref{eq:dynprog})}{=}\sup_{\nu\in\mathcal{P}_1(M)}\left\{\int \phi^f(0,x)\,d\nu-\underline{W}(\mu_0,\nu)\right\}\\\notag
    &\hspace{2mm}= \underline{W}_{\mu_0}^\ast(\phi^f)=-\int \widetilde{\phi^f}\,d\mu_0.
\end{align}
Thus, plugging this into our dual formula (\ref{eq:Bddual}) and restricting our supremum to $\dbound$ gives
\begin{equation*}
    \underline{B}_{\mu_0}(\nu)=\underline{B}_{\mu_0}^{\ast\ast}(\nu)\ge\sup_{f\in \dbound}\left\{\int f\, d\nu+\int \widetilde{\phi^f}\,d\mu_0\right\}.
\end{equation*} 
To show the reverse inequality we will adapt the mollification argument set out in \cite[Proof of Theorem 2.1]{mikami}. We assume our mollifier $\eta_\epsilon(x)$ is such that $\eta_1(x)$ is a smooth function on $[-1,1]^d$ that satisfies $\int \eta_1(x)\,dx=1$ and $\int x\eta_1(x)\,dx=0$, then define $\eta_\epsilon(x)=\epsilon^{-d}\eta_1(x/\epsilon)$. Then for Lipschitz $f$, $f_\epsilon:=f\ast \eta_\epsilon$ is smooth with bounded derivatives. We can derive a bound on $\underline{B}_{\mu\ast\eta_\epsilon}^\ast(f)$ by removing the supremum in \cref{eq:dualexp} and fixing a process $X\in\mathcal{A}$. This gives us:
\begin{multline*}
    \expect{f_\epsilon(X(T))-\int_0^T L(s,X(s),\beta_X(s,X))\,ds-\langle X(0),Y\rangle}\overset{(\text{A2})}{\le}\\
    \expect{f(X(T)+H_\epsilon)-\int_0^T \frac{L(s,X(s)+H_\epsilon,\beta_X(s,X))-\Delta L(0,\epsilon)}{1+\Delta L(0,\epsilon)}\,ds-\langle X(0)+H_\epsilon,Y+H_\epsilon\rangle+\abs{H_\epsilon}^2}\le \\
    \frac{\underline{B}_{\mu_0\ast\eta_\epsilon}^\ast(f\bracket{1+\Delta L(0,\epsilon)})}{1+\Delta L(0,\epsilon)}+T\frac{\Delta L(0,\epsilon)}{1+\Delta L(0,\epsilon)}+d\epsilon^2, 
\end{multline*}
where $H_\epsilon$ is a random variable with distribution $\eta_\epsilon\,dx$ that is independent of $X(\cdot),Y$ (we use this to get the second line), thus $X(T)+H_\epsilon\sim d(\eta_\epsilon\ast \nu_T)$. The third line arises by maximizing over processes $(X(\cdot)+H_\epsilon,Y+H_\epsilon)$.
\par 
Taking the supremum over $X\in\mathcal{A}_{\mu_0}$ of the left side above, we can retrieve a bound on $\underline{B}_{\mu_0}^\ast(f_\epsilon)$. This bound allows us to say
\begin{equation*}\begin{split}
    \int f_\epsilon\,d\nu-\underline{B}_{\mu}^\ast(f_\epsilon)\ge  \int f\,d\nu_\epsilon-\frac{\underline{B}_{\mu_\epsilon}^\ast(f\bracket{1+\Delta L(0,\epsilon)})}{1+\Delta L(0,\epsilon)}+T\frac{\Delta L(0,\epsilon)}{1+\Delta L(0,\epsilon)}+d\epsilon^2,
\end{split}
\end{equation*}
where we use $\epsilon$-subscript to indicate convolution of a measure with $\eta_\epsilon$. Taking the supremum over $f\in\lip{M}$, we get the reverse inequality:
\begin{equation*}\begin{split}
    \sup_{f\in \dbound}\left\{\int f\,d\nu-\underline{B}_{\mu}^\ast(f)\right\}\ge \frac{\underline{B}(\mu_\epsilon,\nu_\epsilon)}{1+\Delta L(0,\epsilon)}+T\frac{\Delta L(0,\epsilon)}{1+\Delta L(0,\epsilon)}+d\epsilon^2\overset{\epsilon\searrow 0}{\ge}\underline{B}(\mu_0,\nu_T),
\end{split}
\end{equation*}
\end{proof}
In some sense $\nabla\phi$ is more fundamental than $\phi$, since our dual is invariant under $\phi\mapsto \phi+c$. Thus when discussing the convergence of a sequence of $\phi$, we refer to the convergence of their gradients. Notably the optimal gradient may not be bounded or smooth, hence may not be achieved within the set $\dbound$. In the subsequent corollary, we denote $\mathcal{P}_X$ the measure on $M\times [0,T]$ associated with the process $X$ and refer to solutions to (\ref{eq:HJB}) by $\phi_n^t(x):=\phi_n(t,x)$ without reference to their final condtion for convenience.
\begin{corollary}[Optimal Processes for $\underline{B}$]\label{cor:minoptX}
Suppose the assumptions on \cref{thm:stocbdual} are satisfied and $d\mu_0\ll d\lambda$. Then $(V,X(t))$ minimizes $\underline{B}(\mu_0,\nu_T)$ iff it is a solution to the SDE
\begin{align}\label{eq:optX}
    dX =& \nabla_p H(t,X, \nabla\phi(t,X))\, dt + dW_t\\\label{eq:optV} 
    V =& \nabla\bar{\phi}(X(0)),
\end{align}
where $\nabla\phi_n(t,x)\rightarrow \nabla\phi(t,x)$ $\mathcal{P}_X$-a.s. and $\nabla\phi_n(0,x)\rightarrow \nabla\bar{\phi}(x)$ $\nu_0$-a.s. for some sequence of $\phi_n$ that approach the supremum in \cref{eq:mindual}. Furthermore $\bar{\phi}$ is concave.
\end{corollary}
\begin{proof}
First note that there exists optimal $(V,X)$, by \cref{thm:interpol}.\par 
We begin with the forward direction. By \cref{thm:stocbdual} $(V,X)$ is optimal iff there exists a sequence of solutions $\phi_n$ to \ref{eq:HJB} that is maximizing in \cref{eq:maxdual} such that
\begin{equation}\label{eq:converge}
    \expect{\int_0^T L(t,X,\beta_X(t,X))\,dt+\langle X(0),V\rangle}=\lim_{n\rightarrow\infty}\expect{\phi_n^T(X(T))+\widetilde{\phi}_n^0(V)}
\end{equation}
We add $0$ to the right hand side to get
\begin{equation}\label{eq:ineqs}
    \lim_{n\rightarrow\infty}\expect{\underbrace{\phi_n^T(X(T))-\phi_n^0(X(0))}_{\text{(a)}}+\underbrace{\phi_n^0(X(0))-\widetilde{\widetilde{\phi}_n^0}(X(0))}_{\text{(b)}}+\underbrace{\widetilde{\widetilde{\phi}_n^0}(X(0))+\widetilde{\phi}_n^0(V)}_{\text{(c)}}}.
\end{equation}
where $\widetilde{\widetilde{f}}$ is the concave double dual of $f$, or its concave hull. Applying Ito's formula to the first two terms, with the knowledge that they satisfy (\ref{eq:HJB}) we get
\begin{equation*}
    \expect{\phi_n^T(X(T))-\phi_n^0(X(0))}=\expect{\int_0^T\langle \beta_X,\nabla\phi_n^t(X(t))\rangle-H(t,X,\nabla\phi_n^t(X(t)))\,dt}
\end{equation*}
However, by the definition of the Hamiltonian, we have $\langle v,b\rangle - H(t,x,v)\le L(t,x,b)$, this along with results from convex analysis provides us with three inequalities on \cref{eq:ineqs}
\begin{align}
   \tag{a} \langle \beta_X,\nabla\phi_n^t(X(t))\rangle-H(t,X,\nabla\phi_n^t(X(t)))\le& L(t,X,\beta_X(t,X))\\
   \tag{b} \phi_n^0(X(0))-\widetilde{\widetilde{\phi}_n^0}(X(0))\le& 0\\
   \tag{c} \widetilde{\widetilde{\phi}_n^0}(X(0))+\widetilde{\phi}_n^0(V)\le& \langle V,X(0)\rangle.
\end{align}
These inequalities demonstrate that \cref{eq:ineqs} breaks our problem into a stochastic and a Wasserstein transport problem (in the flavour of \cref{thm:interpol}), along with a correction term to account for $\phi^0_n$ not being concave. Adding \cref{eq:converge} to the mix, allows us to obtain $L^1$ convergence in the (a,b,c) inequalities, hence a.s. convergence of a subsequence $\phi_{n_k}$. \par  
By properties on convex analysis, convergence in (b,c) is equivalent to $\phi_n^0$ converging $\nu_0$-a.s. to a concave function $\overline{\phi}$ such that $x\mapsto \nabla\overline{\phi}$ is the optimal transport plan for $\underline{W}(\nu_0,\mu_0)$ \cite{Brenier}.
\par
To obtain the optimal control for the stochastic process, one needs the uniqueness of the point $p$ achieving equality in (a). This is a consequence of the strict convexity and coercivity of $b\mapsto L(t,x,b)$ for all $t,x$. The differentiability of $L$ further ensures this value is achieved by $p=\nabla_v L(t,x,b)$. Hence (a) holds iff
\begin{equation*}
    \nabla\phi_n^t(X_t)\longrightarrow \nabla_v L(t,X,\beta_X(t,X))\qquad\mathcal{P}_X\text{-a.s.}
\end{equation*}
Since $\phi_n^t$ are deterministic functions, this demonstrates that $X_t$ is a Markov process with drift $\beta_X$ determined by the inverse transform: $\beta_X(t,X)=\nabla_p H(t,X,\nabla\phi(t,X))$, i.e.,  \cref{eq:optX}.
\end{proof}
\begin{remark}
It is not possible to conclude from the above work that $\bar{\phi}(x)=\phi(0,x)$ since $\bar{\phi}$ is defined on a $\mathcal{P}_X$-null set. To do so, we would require a regularity result on the optimal $\phi$ in time.
\end{remark}

\subsection{Maximizing Ballistic Cost}\label{sec:Bupper}
We now turn our focus to the related stochastic cost:
\begin{equation}\label{eq:Bover}
    \overline{B}(\nu_0,\mu_T):=\sup\left\{\expect{\langle V,X(T)\rangle}-\expect{\int_0^T L(t,X,\beta_X(t,X))\,dt}\middle\rvert X\in \mathcal{A}_{\nu_0}, V\sim\mu_T\right\},
\end{equation}
which we term the \emph{maximizing ballistic cost}.
\begin{remark}
The case where $\mu_T$ is given by a dirac measure $\delta_u$ proves suggestive. Here the maximizing ballistic cost may be interpreted literally in terms of the HJB equation by \cref{prop:HJB}---with $f:x\mapsto \langle u,x\rangle$. Thus, in this particular case we can recover
\begin{equation*}
    \overline{B}(\nu_0,\delta_u)=\int \phi^f(0,x)\,d\nu_0(x),
\end{equation*}
without the aid of any duality theorem. Notably, 
\begin{equation*}
    f^\ast(v)=\sup_z\{\langle v,z\rangle - \langle u,z\rangle\}=\begin{cases}
    0& v=u\\
    \infty & \text{otherwise}.
    \end{cases}
\end{equation*}
\Cref{prop:HJB} further gives us the minimizing process, described by the SDE
\begin{equation*}
    X(t)=X(0)+\int_0^t \nabla_p H(s,X,\nabla \phi(s,X))\,ds+W_t.
\end{equation*}
In the case where $L(t,x,v)=\frac{\abs{v}^2}{2}+\ell(x,t)$, the maximizing ballistic cost may be considered as finding the optimal stochastic process with the final drift constrained a.s. to be $\beta_X(T,X)=u$. In this section we will show that these results generalize to all measures in $\mathcal{P}_1(M^\ast)$.
\end{remark}

We begin by obtaining a similar interpolation result:
\begin{theorem}[Interpolation of $\overline{B}$]\label{thm:Bointerpol}
If the Lagrangian matches assumptions (A), then
\begin{equation}\label{eq:maxinterpol}
    \overline{B}(\nu_0,\mu_T)=\sup\{\overline{W}(\nu,\mu_T)-C_{L}(\nu_0,\nu)\},
\end{equation}
where $C_{L}$ is the action corresponding to the Lagrangian $L$. Furthermore, if $\nu_0\in\mathcal{P}_1(M)$, and $\mu_T\in\mathcal{P}_1(M^\ast)$ there exist an optimal interpolant $\nu_T$. 
\end{theorem}
\begin{proof}
We will proceed as in our proof of the first inequality. First, consider a sequence of rvs $(X_n(\cdot),V_n)$ approximating $\overline{B}(\nu_0,\mu_T)$. Then
\begin{align*}
    \expect{\langle V_n,X_n(T)\rangle}-\expect{\int_0^T L(t,X_n,\beta_{X_n}(t,X_n))\,dt}<\overline{B}(\nu_0,\mu_T)+\frac{1}{n}
\end{align*}
but if $X_n(T)\sim\mu_n$,
\begin{align*}
    \expect{\langle V_n,X_n(T)\rangle}\le&\overline{W}(\mu_n,\mu_T)\\ 
    \expect{\int_0^T L(t,X_n,\beta_{X_n}(t,X_n))\,dt}\ge& C_{L}(\nu_0,\mu_n),
\end{align*}
showing that $\overline{B}(\nu_0,\mu_T)$ is less than its interpolations. To get the reverse inequality, take $\mu_n$ to be a sequence approaching the supremum. Then there exists a random variable $X_n(\cdot)\in\mathcal{A}_{\nu_0}^{\mu_n}$ such that
\begin{equation*}
    \expect{\int_0^T L(t,X_n,\beta_{X_n}(t,X_n))\,dt}<C(\nu_0,\mu_n)+\tfrac{1}{n}.
\end{equation*}
Since $X_n(T)\sim\mu_n$, we may also define $V_n(\omega)\sim \gamma^n_{X_n(0)(\omega)}$ to be distributed according to the disintegration of $d\gamma_x^n(y)\otimes d\mu_T(x)=d\gamma_n$ such that
\begin{equation*}
    \iint\langle y,x\rangle \, d\gamma_n(x,y)>\overline{W}(\mu_n,\mu_T)-\tfrac{1}{n}.
\end{equation*}
We may further assume the minimizing sequence $\mu_n$ is tight in the case where $\nu_0\in \mathcal{P}_1(M^\ast)$ and $\mu_T\in \mathcal{P}_1(M)$ since by \cref{prop:Ccoerc}, \begin{equation*}
    \overline{B}(\nu_0,\mu_T)-1<\overline{W}(\mu_n,\mu_T)-C_{L}(\nu_0,\mu_n)<M\int \abs{v}\,d\mu_n-C_{L}(\nu_0,\mu_n).
\end{equation*}Again, Prokhorov's theorem concludes the proof.
\end{proof}
\begin{theorem}[Dual of $\overline{B}$]\label{thm:Boverdual}
Assume the Lagrangian satisfies the assumptions (A). If $\nu_0\in\mathcal{P}_1(M)$, $\mu_T$ has compact support, and $\overline{B}(\nu_0,\mu_T)<\infty$, then 
\begin{equation}\label{eq:maxdual}
    \overline{B}(\nu_0,\mu_T)=\inf_{g\in\convex \dbound}\left\{\int_{M^*} g^\ast\,d\mu_T+\int_M\phi^{g}\,d\nu_0 \right\},
\end{equation}
where $\phi$ solves the Hamilton-Jacobi-Bellman equation
\begin{align}\label{eq:HJB2}\tag{HJB2}
    \pderiv{\phi}{t}+\frac{1}{2}\Delta \phi -H(t,x,\nabla\phi)=&0&\phi(x,T)=g(x)
\end{align}
\end{theorem}
\begin{lemma}
Let $\mu\in\mathcal{P}_1(M^\ast)$ have compact support. If we define $\overline{W}_\mu(\nu):\nu\mapsto \overline{W}(\mu,\nu)$, then for all $f\in \lip{M^\ast}$
\begin{equation*}
    \sup_{\nu\in\mathcal{P}_1(M)}\left\{\int_M f\,d\nu+\overline{W}(\mu,\nu)\right\}=\int (-f)^\ast\,d\mu.
\end{equation*}
\end{lemma}
\begin{proof}This lemma follows from applying \cref{lem:Wlegendre} to $\underline{W}(\flip{\mu},\nu)$ (recall $\flip{\mu}(A)=\mu(-A)$), and noting that
\begin{equation*}
    -\int \widetilde{f}(v)\,d\flip{\mu}(v)=-\int \widetilde{f}(-v)\,d\mu(v)=\int(-f)^\ast(v)\,d\mu(v).
\end{equation*}
\end{proof}
\begin{proof}[Proof of \cref{thm:Boverdual}]
Define $\overline{B}_{\nu_0}:\mu_T\mapsto\overline{B}(\nu_0,\mu_T)$, and note that this is a concave function (which can be seen by applying \cref{lem:convexlemma} to the negative of \cref{eq:maxinterpol}). Furthermore it is upper semi-continuous on $\mathcal{P}_1(M^\ast)$ by the same reasoning as \cref{cor:Blsc}a). Thus we have
\begin{equation}\label{eq:maxddual}
    \overline{B}_{\nu_0}(\mu_T)=-(-\overline{B}_{\nu_0})^{\ast\ast}(\mu_T)=\inf_{f\in \lip{M^\ast}}\left\{-\int_{M^*} f\,d\mu_T+(-\overline{B}_{\nu_0})^\ast(f)\right\}.
\end{equation}
Investigating the dual, we find
\begin{align}\notag
    (-\overline{B}_{\nu_0})^\ast(f)=&\sup_{\mu_T\in\mathcal{P}_1(M^\ast)} \left\{\int_{M^*} f\,d\mu_T +\overline{B}_{\nu_0}(\mu_T)\right\}\\\notag
    =&\sup_{\substack{\mu_T\in\mathcal{P}_1(M^\ast)\\\nu_T\in\mathcal{P}_1(M^\ast)}} \left\{\int_{M^*} f\,d\mu_T +\overline{W}(\nu_T,\mu_T)-C_{L}(\nu_0,\nu_T)\right\}\\ \label{eq:lipmaxdual}
    =&\sup_{\nu_T\in\mathcal{P}_1(M^\ast)} \left\{\int_{M^*} (-f)^\ast\,d\nu_T -C_{L}(\nu_0,\nu_T)\right\}.
\end{align}
Note that in the case where $(-f)^\ast\in\dbound$, this is simply $\int_M\phi^{(-f)^\ast}\,d\nu_0$, giving us
\begin{equation*}
    \overline{B}_{\nu_0}(\mu_T)\le \inf_{(-f)^\ast\in \dbound}\left\{-\int_{M^*} f\,d\mu_T+\int_M \phi^{(-f)^\ast}\,d\nu_0\right\}.
\end{equation*}
In either case, we can restrict our $f$ to be concave by noting that if we fix $g=(-f)^\ast$, then the set of corresponding $\{-f\rvert (-f)^\ast=g\}$ is minimized by the convex function $g^\ast=(-f)^{\ast\ast}\le -f$ \cite[Proposition 4.1]{ConvexEkeland}. Thus it suffices to consider $f$ convex.\par 
We now show that it is sufficient to consider this infimum over $g\in \convex\dbound$ by a similar mollification argument to \cref{thm:stocbdual} (note that the mollifying preserves convexity). Maintaining the same assumptions and notation as in our earlier argument, we first note a useful application of Jensen's inequality to the legendre dual of a mollified function:
\begin{equation*}\begin{split}
    g_\epsilon^\ast(v)=&\sup_x\left\{\langle v,x\rangle -\expect{g(x+H_\epsilon)}\right\}
    \overset{(\text{J})}{\le}\sup_x\left\{\langle v,x\rangle -g(x)\right\}= g^\ast(v).
\end{split}
\end{equation*}
Mikami \cite[Proof of Theorem 2.1]{mikami} further shows that 
\begin{equation*}
    (\ref{eq:lipmaxdual})=C^\ast_{\nu_0}(g_\epsilon)\le \frac{C^\ast_{\nu_0\ast\eta_\epsilon}((1+\Delta L(0,\epsilon))g)}{1+\Delta L(0,\epsilon)}+T\frac{\Delta L(0,\epsilon)}{1+\Delta L(0,\epsilon)}.
\end{equation*}
Putting these together we get
\begin{equation*}
    \int g^\ast_\epsilon\,d\mu_T+(-\overline{B}_{\nu_0})^\ast(g_\epsilon^\ast)\,d\nu_0\le \int g^\ast\,d\mu_T+\frac{C^\ast_{\nu_0\ast\eta_\epsilon}((1+\Delta L(0,\epsilon))g)}{1+\Delta L(0,\epsilon)}+T\frac{\Delta L(0,\epsilon)}{1+\Delta L(0,\epsilon)}.
\end{equation*}
And once we take the infimum over $g\in\convex\lip{M}$, we get
\begin{equation*}
    \inf_{g\in c\dbound}\left\{\int g^\ast\,d\mu_T+\bracket{-\overline{B}_{\nu}}^\ast(-g^\ast)\right\}\le \frac{-(-\overline{B})^{\ast\ast}_{\nu_0\ast\eta_\epsilon}(\mu_{L,\epsilon})}{1+\Delta L(0,\epsilon)}+T\frac{\Delta L(0,\epsilon)}{1+\Delta L(0,\epsilon)},
\end{equation*}
where $d\mu_{L,\epsilon}(v):=d\mu_T(\bracket{1+\Delta L(0,\epsilon)}v)$. Taking $\epsilon\searrow 0$ dominates the right side by $\overline{B}(\nu_0,\mu_T)$ (where we exploit the upper semi-continuity of $\overline{B}$), completing the reverse inequality.
\end{proof}
\begin{corollary}[Optimal Processes for $\overline{B}$]\label{cor:maxoptX}
Suppose the assumptions on \cref{thm:Boverdual} are satisfied, with $d\mu_0\ll d\lambda$. Then, $(V,X(t))$ is an optimal process iff it is a solution to the SDE
\begin{align}\label{eq:maxoptX}
    dX =& \nabla_p H(t,x,\nabla\phi(t,X))\, dt + dW_t\\\label{eq:maxoptV} 
    X(T) =& \nabla\bar{\phi}(V,T)
\end{align}
where $\lim_{n\rightarrow \infty}\phi_n(T,x)\rightarrow \bar{\phi}(x)$ $\nu_T$-a.s. and $\lim_{n\rightarrow\infty}\phi_n(t,x)=\phi(t,x)$ $\mathcal{P}_X$-a.s. for some sequence of $\phi_n$ that is minimizing the dual problem.
\end{corollary}
\begin{proof}
If $(V,X)$ is optimal, then \cref{thm:Boverdual} means there exists a sequence of solutions $\phi_n$ to \ref{eq:HJB} with convex final condition such that
\begin{equation}\label{eq:maxconverge}
    \expect{\langle X(T),V\rangle-\int_0^T L(t,X,\beta_X(t,X))\,dt}=\lim_{n\rightarrow\infty}\expect{\bracket{\phi_n^T}^\ast(V)+\phi_n^0(X(0))}
\end{equation}
We add $0$ to the right hand side to get
\begin{equation*}
    \lim_{n\rightarrow\infty}\expect{\bracket{\phi_n^T}^\ast(V)+\phi_n^T(X(T))-\phi_n^T(X(T))+\phi_n^0(X(0))}.
\end{equation*}
Applying Ito's formula to the last two terms, with the knowledge that $\phi_n$ satisfies (\ref{eq:HJB}) we get
\begin{equation*}
    \expect{-\phi_n^T(X(T))+\phi_n^0(X(0))}=\expect{\int_0^T-\langle \beta_X,\nabla\phi_n^t(X(t))\rangle+H(t,X,\nabla\phi_n^t(X(t)))\,dt}
\end{equation*}
However, by the definition of the Hamiltonian, we have $-\langle v,b\rangle + H(t,x,v)\ge -L(t,x,b)$, similarly $\phi^\ast(v)+\phi(x)\ge\langle v,x\rangle$. These inequalities allow us to separate the limit in \cref{eq:maxconverge} into two requirements: (a) $\langle \beta_X,\nabla\phi_n^t(X(t))\rangle-H(t,X,\nabla\phi_n^t(X(t)))$ must converge to $L(t,X,\beta_X(t,X))$ and (b) $\phi_n^T(X(T))+\bracket{\phi_n^T}^\ast(V)$ must converge to $\langle X(T),V\rangle$ in $L^1$ hence a subsequence $\phi_{n_k}$ exists such that this convergence is a.e.\par  
The journey from (a) to \cref{eq:maxoptX} is as in \cref{cor:minoptX}. The only difference from the earlier corollary is that we know that $\phi_n$ must converge to a convex function, so (b) implies $V=\nabla\lim_{n\rightarrow \infty}\phi_n(X(T))$.

\end{proof}
\begin{remark}
Like for the minimizing cost, it is impossible to conclude $\phi(T,x)=\bar{\phi}(x)$ without a regularity result. \textbf{If} one could make this connection,  then one could formulate \cref{eq:maxoptV} as $V=\nabla \phi(T,X(T))$, allowing us to interpret the measure $\mu_T$ as a condition on our final momentum distribution---i.e.,  our final drift $\beta(T)$ would be constrained to be distributed according to the measure $\rho$ where $\mu_T=\bracket{\nabla L(t,x,\cdot)}_\#\rho$. In the particular case where $L(t,x,v)=\frac{\abs{v}^2}{2}+\ell(x,t)$, this would impose a Neumann type boundary condition on our stochastic process: $\beta(T)\sim\mu_T$.
\end{remark}

\bibliographystyle{plainnat}

\end{document}